\documentclass[12pt,a4paper]{amsart}

\title{Centers of Hecke Algebras of Complex Reflection Groups}
\author{Eirini Chavli}
\address{E.C.: Institute of Algebra and Number Theory, University of Stuttgart, Stuttgart, Germany}
\email{eirini.chavli@mathematik.uni-stuttgart.de}

\author{G\"otz Pfeiffer}
\address{G.P.: School of Mathematical and Statistical Sciences,
  National University of Ireland,
  Galway, University Road, Galway, Ireland}
\email{goetz.pfeiffer@nuigalway.ie}

\makeatletter
\@namedef{subjclassname@2020}{%
  \textup{2020} Mathematics Subject Classification}
\makeatother

\subjclass[2020]{Primary 20C08; Secondary 20F55}

\usepackage[euler-digits]{eulervm}
\usepackage[margin=1in]{geometry}
\usepackage{hyperref}
\usepackage{enumitem}
\usepackage{blkarray}

\usepackage{tikz}
\tikzset{every node/.style={circle,draw=none,minimum size=4mm,inner sep=0pt}}
\tikzset{r/.style={red,->,>=latex}}%
\tikzset{b/.style={blue,->,>=latex}}%

\numberwithin{equation}{section}

\newtheorem{Theorem}{Theorem}[section]
\newtheorem{Proposition}[Theorem]{Proposition}
\newtheorem{conj}[Theorem]{Conjecture}
\newtheorem{rem}[Theorem]{Remark}
\newtheorem{ex}[Theorem]{Example}

\DeclareMathOperator{\Hom}{Hom}
\DeclareMathOperator{\Irr}{Irr}
\DeclareMathOperator{\Cl}{Cl}

\begin{document}
	%!TeX spellcheck = en_us
\begin{abstract}
  We provide a dual version of the Geck--Rouquier
  Theorem~\cite{GeckRouquier97} on the center of an Iwahori--Hecke
  algebra, which also covers the complex case.  For the eight complex
  reflection groups of rank $2$, for which the symmetrising trace
  conjecture is known to be true, we provide a new faithful matrix model for their
  Hecke algebra $H$.
These models enable concrete calculations inside $H$.
  For each of the eight groups, we compute an explicit integral basis
  of the center of $H$.
\end{abstract}

\maketitle

\section{Introduction}
Let $W$ be a finite complex reflection group and $H$ the associated generic Hecke algebra, defined over the Laurent polynomial ring $R=\mathbb{Z}[u_1^{\pm},\dots, u_k^{\pm}]$, where $\{u_i\}_{1\leq i\leq k}$ is a set of parameters whose cardinality depends on $W$. In 1999,  Malle  \cite[\S 5]{malle1} proved that $H$ is split semisimple when defined over the field $F=\mathbb{C}(v_1,\dots, v_k)$, where each parameter $v_i$ is a root of $u_i$ of rank $N_W$, for some specific $N_W\in \mathbb{N}$. By Tits' deformation theorem \cite[Theorem 7.4.6]{GP}, the specialization $v_{j} \mapsto 1$ induces a bijection
$\Irr(H\otimes_RF)\rightarrow \Irr(W)$.

A natural question is how the irreducible representations behave after specializing the parameters $u_i$ to arbitrary complex numbers. If the specialized Hecke algebra is semisimple, Tits' deformation theorem still applies; the simple representations of the specialized Hecke algebra are parametrized again by $\Irr(W)$. However, if the specialized algebra is not semisimple one needs to find another way to parametrise the irreducible representations. One main obstacle in this direction is the lack of the description of the center $Z(H)$ of the Hecke algebra $H$ (see  \cite[Lemma 7.5.10]{GP}).

Apart from the real case \cite{GeckRouquier97}, there is not yet a known precise description of the center of the generic Hecke algebra $H$ in the complex case, except for the groups $G_4$ and $G(4, 1, 2)$ provided by Francis \cite{fr}.
In this paper, we introduce a new general method for computing an $R$-basis of the center of $H$.
Let $\mathcal{B}$ be an $R$-basis of $H$. Expressing an arbitrary element $z \in Z(H)$ as a linear combination of $\mathcal{B}$, the conditions $sz = zs$, one for each generator $s$ of $H$, give an $R$-linear system, whose solution describes a basis of $Z(H)$ as linear combinations of the elements of $\mathcal{B}$. This elementary approach has the following three difficulties:
\begin{itemize}
	\item Calculations inside the Hecke algebra are very complicated, even for products of the form $sb$ and $bs$, for $b \in \mathcal{B}$. In the last four years, there is a progress on this direction (see, for example, \cite{BCCK, BCC, CC}). However, it is still unclear if one can completely automate such calculations, since all the attempts so far use  a lot of (long) computations made by hand.
	\item Solving the aforementioned $R$-linear system is not always easy.
	%, especially when the groups become larger.
	It is known that Gaussian elimination can suffer from coefficient explosion
	over
	a field of rational functions in several variables, such as the quotient
	field of the ring $R$ of Laurent polynomials.
	As a result,
	the choices of dependent and independent variables can be crucial, since wrong choices can lead to dead ends.
	\item Even if the $R$-linear system can be solved (over the fraction
	field), the solution cannot be expected to lie in $R$. In fact, for the case of $G_{12}$ a first attempt solving this system provided us with a solution not in $R$ (details can be found in the project's webpage \cite{pr}).
\end{itemize}

However, we show that it is possible to overcome all these difficulties, as
demonstrated by our results for particular examples.
A natural first example is the smallest exceptional group, the group $G_4$.
For this group we use a new approach on constructing an
$R$-basis $\mathcal{B}$ (see Example \ref{exa}), which allows us
to work with a faithful matrix representation of $H$, rather than its usual presentation.  Hence, we can
\begin{itemize}
	\item  automate
	calculations inside the Hecke algebra of $G_4$ and, in particular,
	compute the elements $sb$ and $bs$ from above,
	\item solve the $R$-linear system $sz = zs$,
	and 
	% now 
	% The fact that
	% the group $G_4$ is of order 24 (and, hence, admits an $R$-basis of
	% 24 elements) and that the Laurent polynomial ring $R$ depends only
	% on three parameters, allowed us to solve the $R$-linear system that
	% gives us an $R$-basis of the center of the associated Hecke algebra,
	% where each element is expressed as $R$-linear combination of the
	% elements of the basis of the Hecke algebra.
	\item verify that all solutions we obtain for $z$ are in fact
	$R$-linear combinations of $\mathcal{B}$.
\end{itemize}

Our next goal is to explain the coefficients of these $R$-linear
combinations.  This allows us to describe the center $Z(H)$ for other
complex reflection groups without relying on the solution of the
$R$-linear system.  In order to explain our findings, we first need to
revisit the real case.

Let $W$ be a real reflection group and $H$ its associated Iwahori--Hecke algebra, which admits a standard basis $\{T_w\,:w\in W\}$.
Denote by $\Cl(W)$ the set of conjugacy classes of $W$, and choose a
set of representatives $\{w_C \in C \mid C \in \Cl(W)\}$ such that each
element $w_C$ has minimal length in its class $C$.  It can
be shown that there exist uniquely determined polynomials
$f_{w,C} \in R$, independent of the choice of the
minimal length representatives $w_C$, the so-called \emph{class
	polynomials} (see \cite[\S 8.2]{GP}), such that
\begin{align*}
	\chi(T_w) = \sum_{C \in \Cl(W)} f_{w,C}\, \chi(T_{w_C})
\end{align*}
for all $\chi \in \Irr(W)$.  In other words, the column
$\bigl(\chi(T_w)\bigr)_{\chi}$ of character values of $T_w$ is an $R$-linear
combination of the columns $\bigl(\chi(T_{w_C})\bigr)_{\chi}$ of the basis
elements of $H$ corresponding to the conjugacy class representatives
of minimal length.

Clearly, for any choice $\{v_C \in C \mid C \in \Cl(W)\}$ of conjugacy
class representatives, the square matrix $\bigl(\chi(T_{v_C})\bigr)_{\chi, C}$
of character values is invertible as it specializes to the character
table of $W$.  Hence, for each $w \in W$, there are uniquely determined
coefficients $\zeta_{w,C}$ such that
\begin{align*}
	\chi(T_w) = \sum_{C \in \Cl(W)} \zeta_{w,C}\, \chi(T_{v_C}).
\end{align*}
However, the coefficients $\zeta_{w,C}$ cannot be expected to belong to $R$. That is why it is crucial to choose minimal length class representatives.

We denote by $\{T_w^{\vee}:\;w\in W\}$ the dual basis of $\{T_w\,:w\in W\}$ with respect to the standard symmetrising  form (for the definition of the dual basis see, for example, \cite[Definition 7.1.1]{GP}). 
The following theorem has been shown by Geck and Rouquier (see \cite[\S 5.1]{GeckRouquier97} or \cite[Theorem 8.2.3 and Corollary 8.2.4]{GP}).
\begin{Theorem}\label{thm:GRmin} Let $W$ be a finite real reflection group. The elements
	\[
	y_C = \sum_{w \in W} f_{w, C}\, T_w^{\vee}, \quad C \in \Cl(W),
	\]
	form a basis of the center $Z(H)$.
\end{Theorem}

We now examine the complex case. We first assume that the Hecke algebra $H$ admits a symmetrising trace $\tau$. Let $\{b_w:\; w\in W\}$ be a basis of $H$ as $R$-module and let $\{b_w^{\vee}:\; w\in W\}$ be its dual basis with respect to $\tau$. The lack of length function of a complex reflection group $W$ cannot guarantee us the existence of  class polynomials
in the sense of the real case.  However, it remains true that
with respect to any choice of conjugacy class representatives $v_C$,
for each element $w \in W$ there are coefficients $\zeta_{w,C}\in F$, which express the column $(\chi(b_w))_{\chi}$ of character values
as a linear combination of the columns $\bigl(\chi(b_{v_C})\bigr)_{\chi}$.
In this paper, we prove the following:

\begin{Theorem}\label{thm:GRmax} Let $W$ be a finite complex reflection group. For any  choice of conjugacy class representatives $v_C$, the elements
	\[
	y_C = \sum_{w \in W} \zeta_{w, C}\, b_w^{\vee}, \quad C \in \Cl(W),
	\]
	form a basis of the center $Z(H\otimes_RF)$.
\end{Theorem}

This theorem generalizes the theorem~\ref{thm:GRmin} of Geck-Rouquier
in such a way that it includes the complex case, at the expense of working over
$F$ in place of $R$.  At the same time, it gains us some flexibility
in terms of choosing the conjugacy class representatives.
There is a dual version of the theorem that provides even more flexibility, as follows.

% Our hope was that for the case of $G_4$ there is a choice of class representatives $v_C$ such that the coefficients $\zeta_{w, C}$ coincide with the ones we found after solving the $R$-linear system, obtaining from the  conditions $sz = zs$. However, this was not the case. As a result, we noticed that these coefficients are obtained when replacing  the role of the basis and dual basis in Theorem \ref{thm:GRmax}. This result is not  true only for $G_4$, but for each complex reflection group:

For each class $C \in Cl(W)$, we choose again a representative  $v_C \in C$ and we define
coefficients $g_{w,C} \in F$ by the condition
\begin{align*}
	\chi(b_w^{\vee}) = \sum_{C \in \Cl(W)} g_{w, C}\, \chi(b_{v_C}^{\vee}),\;\text{ for all } \chi \in \Irr(H\otimes_RF).
\end{align*}

\begin{Theorem}\label{thm:complex-center} Let $W$ be a finite complex reflection group. For any  choice of conjugacy class representatives $v_C$, the elements
	\[
	z_C = \sum_{w \in W} g_{w, C}\, b_w, \quad C \in \Cl(W),
	\]
	form a basis of the center $Z(H\otimes_RF)$.
\end{Theorem}

As an illustration, and to state the fact that for the real case Theorems \ref{thm:GRmax} and \ref{thm:complex-center} give new bases, different from the one of Geck-Rouquier, we apply these theorems to the Coxeter group of type $A_2$
and express the resulting basis of $Z(H)$ as a linear combination of the
standard basis $\{T_w : w \in W\}$ of $H$.

\begin{ex}
	Let $W$ be a finite Coxeter group of type $A_2$ with generators $s, t$ such
	that $sts = tst$.  Then $\{1, s, st\}$ is a set of minimal length class
	representatives of $W$, and $\{1, sts, st\}$ is a set of maximal length
	class representatives.  The Iwahori--Hecke algebra $H$ is generated by
	elements $T_s$ and $T_t$ and defined over the ring
	$R = \mathbb{Z}[u_1^{\pm}, u_2^{\pm}]$. We set $c = u_1 + u_2$ and
	$d = - u_1 u_2$ (where only $d$ is invertible in $R$).  Then
	$T_s^2 = c\, T_s + d$ and $T_t^2 = c\, T_t + d$.
	% We apply the above theorems and express the resulting basis of the center of $H(W)$ as a linear combination of the standard basis $\{T_w : w \in W\}$ of $H(W)$.
	We have the following:
	\begin{itemize}
		\item Theorem~\ref{thm:GRmin}, with minimal length class representatives %$w_C = 1, s, st$
		yields the basis
		\begin{align*}
			T_1, && d^{-1} (T_s + T_t) + d^{-2}\, T_{sts}, && d^{-2} (T_{st} + T_{ts}) + cd^{-3}\, T_{sts}
		\end{align*}
		\item Theorem~\ref{thm:GRmax}, with maximal length class representatives %$w_C = 1, sts, st$
		yields the basis
		\begin{align*}
			T_1, && d^{-2} (T_s + T_t) + d^{-3}\, T_{sts}, && -cd^{-2} (T_s + T_t) + d^{-2} (T_{st} + T_{ts})
		\end{align*}
		\item Theorem~\ref{thm:complex-center}, with minimal length class representatives %$w_C = 1, s, st$
		yields the basis
		\begin{align*}
			T_1, && (T_s + T_t) + d^{-1}\, T_{sts}, && (T_{st} + T_{ts}) + cd^{-1}\, T_{sts}
		\end{align*}
		\item Theorem~\ref{thm:complex-center}, with maximal length class representatives %$w_C = 1, sts, st$
		yields the basis
		\begin{align*}
			T_1, && d (T_s + T_t) + T_{sts}, && -c (T_s + T_t) + (T_{st} + T_{ts})
		\end{align*}
	\end{itemize}
	Note that in all cases the coefficients of the basis elements of $Z(H)$ belong to $R  = \mathbb{Z}[u_1^{\pm}, u_2^{\pm}]$, and that in the last case
	they even lie in the polynomial ring $\mathbb{Z}[u_1, u_2]$.\qed
\end{ex}
It is worth mentioning here that
Theorems \ref{thm:GRmax} and \ref{thm:complex-center} have a general proof, which does not use the case-by-case analysis,
based on the classification of complex reflection groups \cite{ShTo}.

As mentioned before, Theorems \ref{thm:GRmax} and \ref{thm:complex-center} provide bases of the center of $H$ over the splitting field $F$. Choosing an arbitrary basis of the Hecke algebra and arbitrary class representatives, one cannot expect to obtain a basis of the center of $H$ over $R$.
% \textcolor{red}{counterexample}
In fact, for $G_{12}$, there is a choice of conjugacy class representatives $v_C$, where not all the coefficients $g_{w, C}$ of Theorem~\ref{thm:complex-center}
belong to $R$ (for details, see the project's web page~\cite{pr}).
However, for some choice of a basis $\{b_w:\;w\in W\}$ and of class representatives $v_C$, it might turn out that the coefficients $g_{w,C}$ of Theorem \ref{thm:complex-center}  belong to $R$, which means that we obtain a basis of the center $Z(H)$.

In fact, in this paper we show that for the groups
$G_4,\dots, G_8, G_{12}, G_{13}, G_{22}$, i.e., for all exceptional
groups of rank 2, whose associated Hecke algebra is known to be symmetric, we can
make such choices so that we obtain a basis of the center $Z(H)$. We
conjecture that this is true for all complex reflection groups that
satisfy the symmetrising trace conjecture~\ref{BMM sym}.

The next section of this paper explains in detail how we
make these choices.
The final section contains our main results. In our calculations we
used some programs written in GAP, which one can find in the project's
webpage \cite{pr}.

\bigskip\noindent
\textbf{Acknowledgements.}
Work on this project started during the 2-Day Meeting in Stuttgart on Computational Lie Theory in July 2018, supported by DFG (SFB-TRR195).
The authors would like to thank Gunter Malle and Ivan Marin for useful
comments and discussions.

\section{Choosing a Basis} \label{table}

\subsection{Hecke algebras} \label{Hecke}
A complex reflection group $W$ is a finite subgroup of $\mathrm{GL}_n(\mathbb{C})$ generated by \emph{pseudo-reflections} (these are non-trivial elements of $W$ whose fixed points in $\mathbb{C}^n$ form a hyperplane). Real reflection groups, also known as \emph{finite Coxeter groups} are particular cases of complex reflection groups.

We denote by $K$ the \emph{field of definition} of $W$, that is the field generated by the traces on $\mathbb{C}^n$ of all the elements of $W$.  If $K \subseteq \mathbb{R}$, then $W$ is a finite Coxeter group, and
if $K=\mathbb{Q}$, then $W$ is a Weyl group.

A complex reflection group $W$ is \emph{irreducible} if it acts irreducibly on $\mathbb{C}^n$ and, if that is the case, we call $n$ the  \emph{rank} of $W$. Each complex reflection group is a direct product of irreducible ones and, hence, the study of reflection groups reduces to the irreducible case. The classification of irreducible complex reflection groups is due to  Shephard and Todd \cite{ShTo} and it is given by the following theorem.

\begin{Theorem}\label{ShToClas} Let $W \subset \mathrm{GL}_n(\mathbb{C})$ be an irreducible complex
	reflection group. Then, up to conjugacy, $W$ belongs to precisely one of the following classes:
	\begin{itemize}
		\item The symmetric group $S_{n+1}$.
		\item The infinite family $G(de,e,n)$, where  $d,e,n \in \mathbb{N}^*$, such that  $(de,e,n)\neq (1,1,n)$ and $(de,e,n)\neq(2,2,2)$, of all
		$n\times n$ monomial matrices whose non-zero entries are ${de}$-th roots of unity, while the product of all non-zero
		entries is a $d$-th root of unity.
		\smallbreak
		\item The 34 exceptional groups
		$G_4, \, G_5, \, \dots,\, G_{37}$ (ordered with respect to increasing rank).
	\end{itemize}
\end{Theorem}

Let $W$ be a complex reflection group. We denote by $B$ the   \emph{complex braid group} associated to $W$, as defined in \cite[\S 2 B]{BMR}.
A pseudo-reflection $s$ is called \emph{distinguished} if its only nontrivial eigenvalue on $\mathbb{C}^n$ equals $\exp(-2\pi \sqrt{-1}/e_s)$, where  $e_s$ denotes the order of $s$ in $W$. Let $S$ denote the set of the distinguished pseudo-reflections of $W$.

For each $s\in S$ we choose a set of $e_s$ indeterminates $u_{s,1},\dots, u_{s,e_s}$, such that $u_{s,j}=u_{t,j}$ if $s$ and $t$ are conjugate in $W$. We denote by $R$ the Laurent polynomial ring $\mathbb{Z}[u_{s,j},u_{s,j}^{-1}]$. The \emph{generic Hecke algebra} $H$ associated to $W$ with parameters $u_{s,1},\dots, u_{s,e_s}$ is the quotient of the group algebra $R[B]$ of $B$ by the ideal generated by the elements of the form
\begin{equation}
	(\sigma-u_{s,1})(\sigma-u_{s,2})\cdots (\sigma-u_{s,e_s}),
	\label{Hecker}
\end{equation}
where $s$ runs over the conjugacy classes of $S$ and $\sigma$ over the set of \emph{braided reflections} associated to the pseudo-reflection $s$ (for the standard notion of a braided reflection associated to $s$ one can refer to \cite[\S 2 B]{BMR}).
It is enough to choose one relation of the form described in (\ref{Hecker}) per conjugacy class, since the corresponding braided reflections are conjugate in $B$.

We obtain an equivalent definition of $H$ if we expand the relations \eqref{Hecker}.
More precisely, $H$ is the quotient of the group algebra $R[B]$  by the elements of the form
\begin{equation}
	{\sigma }^{e_{s}}-a_{{s},e_{s}-1}{\sigma}^{e_{s}-1}-a_{{s},e_{s}-2}{\sigma}^{e_{s}-2}-\dots-a_{{s},0},
	\label{Hecker2}
\end{equation}
where $a_{{s},e_{s}-k}:=(-1)^{k-1}f_k({u}_{{s},1},\dots,{u}_{{s},e_{s}})$ with $f_k$ denoting the $k$-th elementary symmetric polynomial, for $k=1,\ldots,e_{s}$.
Therefore, in the presentation of $H$, apart from
the \emph{braid relations} coming from the presentation of $B$, we also have the \emph{positive Hecke relations}:
\begin{equation}
	{\sigma}^{e_{s}}=a_{{s},e_{s}-1}{\sigma}^{e_{s}-1}+a_{{s},e_{s}-2}{\sigma}^{e_{s}-2}+\dots+a_{{s},0}.
	\label{ph}
\end{equation}
We notice now that $a_{{s},0}= (-1)^{e_{s}-1}{u}_{{s},1}{u}_{{s},2}\dots {u}_{{s},e_{s}} \in R^{\times}$. Hence,  ${\sigma}$ is invertible in $H$ with
\begin{equation}\label{invhecke}
	{\sigma }^{-1}=a_{s,0}^{-1}\,{\sigma }^{e_{s}-1}-a_{s,0}^{-1}\,a_{{s},e_{s}-1}{\sigma }^{e_{s}-2}-a_{s,0}^{-1}\,a_{{s},e_{s}-2}{\bf s}^{e_{s}-3}-\dots- a_{s,0}^{-1}\,a_{{s},1}.
\end{equation}
We call  relations \eqref{invhecke} the \emph{inverse Hecke relations}.

If $W$ is a real reflection group, $H$ is known as the \emph{Iwahori--Hecke algebra} associated with $W$ (for more details about Iwahori--Hecke algebras one may refer, for example, to \cite[\S 4.4]{GP}). Iwahori--Hecke algebras admit a \emph{standard basis} $(T_w)_{w \in W}$ indexed by the elements of $W$ (see \cite[IV, \S 2]{Bou05}).
Brou\'e, Malle and Rouquier conjectured a similar result for complex reflection groups \cite[\S 4]{BMR}:

\begin{conj}[The BMR freeness conjecture]\label{BMR free}
	The algebra $H$ is a free $R$-module of rank  $|W|$.
\end{conj}

This conjecture is now a theorem, thanks to work of several people who used a case-by-case analysis approach, in order to prove the case of all irreducible complex reflection groups. A detailed state of the art of the proof can be found in \cite[Theorem 3.5]{BCCK}.

A symmetrising trace on a free algebra is a trace map $\tau$ that induces a non-degenerate bilinear form, meaning that the determinant of the matrix $(\tau(bb'))_{b,b'\in \mathcal{B}}$ is a unit in the ring over which we define the algebra for some (and hence every) basis $\mathcal{B}$ of the algebra. 
For Iwahori--Hecke algebras there exists a unique symmetrising trace, given by $\tau(T_w)=\delta_{1w}$  \cite[IV, \S 2]{Bou05}.
Brou\'e, Malle and Michel  conjectured the existence of a symmetrising trace also for non-real complex reflection groups \cite[\S2.1, Assumption 2(1)]{BMM}:

\begin{conj}[The BMM symmetrising trace conjecture]
	\label{BMM sym}
	There exists a linear map $\tau: H\rightarrow R$ such that:
	\begin{itemize}
		\item[$(1)$] $\tau$ is a symmetrising trace, that is, we have $\tau(h_1h_2)=\tau(h_2h_1)$ for all  $h_1,h_2\in H$, and the bilinear map $H\times H\rightarrow R$,  $(h_1,h_2)\mapsto\tau(h_1h_2)$ is non-degenerate. \smallbreak
		\item[$(2)$] $\tau$ becomes the canonical symmetrising trace on $K[W]$ when ${u}_{s,j}$ specialises to ${\rm exp}(2\pi \sqrt{-1}\; j/e_s)$ for every $s\in S$ and $j=1,\ldots, e_s$. \smallbreak
		\item[$(3)$]  $\tau$ satisfies
		\begin{equation}\label{extra}
			\tau(T_{b^{-1}})^* =\frac{\tau(T_{b\boldsymbol{\pi}})}{\tau(T_{\boldsymbol{\pi}})}, \quad \text{ for all } b \in B(W),
		\end{equation}
		where $b\mapsto T_{b}$ denotes the restriction of the natural surjection $R[B] \rightarrow H$ to $B$,
		$x \mapsto x^*$ is the automorphism of  $R$ given by $u_{s,j} \mapsto u_{s,j}^{-1}$ and $\boldsymbol{\pi}$ the element $z^{|Z(W)|}$, with $z$ being the image of a suitable generator of the center of $B$ inside $H$.
		\smallbreak
	\end{itemize}
\end{conj}
Since we have the validity of the BMR freeness conjecture, we know  \cite[\S 2.1]{BMM} that if there exists such a linear map $\tau$, then it is unique. If this is the case, we call $\tau$ the \emph{canonical symmetrising trace on}  $H$.

Malle and Michel \cite[Proposition 2.7]{MalleMichel} proved that if the Hecke algebra $H$ admits a basis $\mathcal{B}\subset B$ consisting of braid group elements that satisfies  certain properties (among them that $1\in \mathcal{B}$), then Condition \ref{extra}
is equivalent to:
\begin{equation}\label{extra2}
	\tau(T_{x^{-1}\boldsymbol{\pi}})=0, \quad \text{ for all } x \in \mathcal{B} \setminus \{1\}.
\end{equation}

Apart from the real case, the BMM symmetrising trace conjecture is known to hold  for a few exceptional groups and for the infinite family (detailed references can be found in \cite[Conjecture 3.3]{CC}).

We now describe the representation theory of Hecke algebras. In \cite[\S 5]{malle1} Malle associates to each complex reflection group $W$ a positive integer  $N_W$ and he defines for each $s\in S$,  a set of $e_s$ indeterminates $v_{s,1},\dots, v_{s,e_s}$
by the property	\begin{equation}\label{split}
	v_{s,j}^{N_W}=\exp(-2\pi \sqrt{-1}\; j/e_s)u_{s,j}.
\end{equation} We denote by $F$ the field $\mathbb{C}(v_{s,j})$ and
by extension of scalars we obtain the algebra
$FH:=H\otimes_{R}F$, which is split semisimple (\cite[Theorem 5.2]{malle1}). By Tits' deformation theorem \cite[Theorem 7.4.6]{GP}, the specialization $v_{s,j} \mapsto 1$ induces a bijection
Irr$(FH)\rightarrow$ Irr$(W)$.

Models of irreducible representations of the Hecke algebra $FH$ associated with certain irreducible
complex reflection groups have been computed by Malle and
Michel~\cite{MalleMichel}, and are readily available in Jean Michel's
development version~\cite{chevie-jm} of the CHEVIE package~\cite{chevie}.
In this paper, we use these models to evaluate the irreducible characters
on some particular elements, when constructing an explicit basis of the center of
the Hecke algebra.

\subsection{Coset table}
Let $W$ be a complex reflection group with associated Hecke algebra $H$ defined over $R$. The goal of this paper is to provide, at least in some examples, a basis of the center $Z(H)$ of $H$ as $R$-module. For this purpose, we first find a basis $\mathcal{B}$ of the Hecke algebra and then we describe the basis elements of $Z(H)$ as linear combinations of elements of $\mathcal{B}$.

In this section we explain the method we use in order to find a basis $\mathcal{B}$ for  the exceptional groups $G_4,\,\dots,\, G_8,\, G_{12},\, G_{13},\, G_{22}$. These groups are exactly the exceptional
complex reflection groups of rank 2 for which we know, apart from the validity of the BMR freeness conjecture \ref{BMR free},  the validity of the BMM symmetrising trace conjecture \ref{BMM sym} as well.
Our method is not an algorithm and we cannot be sure it works in general since, as we will see in a while, one needs to make some crucial choices, which are a product of experimentation and experience. At the end of this section we give in detail the example of $G_4$, where the reader can see thoroughly our methodology and arguments.

We recall that $W$ is generated by distinguished pseudo-reflections $s$. We choose a particular generator $s_0$ of $W$ and we denote by $W'$ the parabolic subgroup of $W$ generated by $s_0$ and by $H'$ the subalgebra of $H$ generated by $\sigma_0$.
The action of $W$ on the cosets of $W'$
defines a graph with vertex set $\{W' x \mid x \in W\}$
and edges $W' x \stackrel{s}{\longrightarrow} W'xs$, for $s$ running over the generators of $H$. In the project's webpage \cite{pr} the reader can find these graphs for the examples we are dealing with in this paper.

We now choose class representatives $x_i$ as follows: We choose some representatives  as the \emph{anchor coset representatives}, and we pick representatives for the
remaining cosets along the spanning tree. We always choose  $x_1=1$.

The coset representatives $x_i$ are in fact explicit words in generators of $W$. These generators are in one to one correspondence with generators of the Hecke algebra $H$, by sending $s$ to $\sigma$. Hence, we can obtain from the elements $x_i$ corresponding elements inside the Hecke algebra, which we also denote by $x_i$.

Our goal now is to prove that this chosen set
$\{x_i\,:\,i=1,\dots,|W/W'|\}$ is a basis of $H$ as $H'$-module.
Since $H$ is a free $H'$-module of dimension $|W/W'|$ (see, for example, \cite{CC}), we only have to prove that $\{x_i\}$ is a spanning set for $H$. By construction, we always have $x_1=1$ and, hence, it is enough to prove that for each  $x_i$,   the elements $x_i.\sigma$ are  linear combinations of the form
$\sum_i h_i \cdot x_i$, where $h_i \in H'$.

In order to prove that, we construct a \emph{coset table}, where we list this linear combination not only for the elements $x_i.\sigma$, but also for the elements $x_i.\sigma^{-1}$. The reason we calculate these extra linear combinations is because they are prerequisite to calculating the linear combinations for the elements $x_i.\sigma$.

In order to fill the coset table we use a program created in GAP. This program uses the positive and inverse Hecke relations \eqref{ph} and \eqref{invhecke} and also some simple hand-calculations. On the project's webpage \cite{pr} one can find these programs for the aforementioned exceptional groups of rank 2.

The completed coset table and the fact that $H'$ is a free $R$-module provides a basis for the Hecke algebra $H$ over $R$ as follows:
$\mathcal{B}=\{x_i,\; \sigma_0x_i,\dots, \sigma_0^{e_s-1}x_i\,:\,i=1,\dots,|W/W'|\}$.

Our goal now is to express every element of the Hecke algebra as $R$-linear combination of the elements of  $\mathcal{B}$. In order to do that, we use again the completed coset table. The expression of $x_i.\sigma$ as linear combination of the form
$\sum_i h_i \cdot x_i$, $h_i \in H'$ yields a representation $\rho$
for $H$ as a free module over the subalgebra~$H'$. We use the  matrix models of this representation in order to compute the image of each word $x$ in generators  $\sigma$ (and their
inverses) inside the Hecke algebra $H$ as the image of the vector
\begin{align*}
	1_H = 1_{H'} x_1 + 0 x_2 + 0x_3+\dots + 0 x_{|W/W'|} = (1_{H'}, 0,0,\dots,0)
\end{align*}
under the product of the matrices corresponding to the letters of the word $x$.  We can see this as follows: Let $x$ be the word $\sigma_{i_1}^{m_1}\cdots\sigma_{i_r}^{m_r}$, where $m_1,\dots,m_r\in \mathbb{Z}^{*}$. Then the image of $x$ inside $H$ can be computed as the product $1_H\cdot\rho(\sigma_{i_1})^{m_1}\cdots\rho(\sigma_{i_r}^{m_r})$. This is a vector with coefficients in $H'$ and, hence, it corresponds to an $H'$-linear combination of $x_i$'s and, hence, to an $R$-linear combination of elements in $\mathcal{B}$.

In order to make this clearer to the reader, we give the following example of the exceptional group $G_4$:

\begin{ex}\label{exa}
	Let $G_4=\langle s_1, s_2 \mid s_1^3=s_2^3=1, \;s_1s_2s_1=s_2s_1s_2\rangle.$
	The Hecke algebra $H$ of $W$ is defined over $R=\mathbb{Z}[u_1^{\pm}, u_2^{\pm}, u_3^{\pm}]$ and it admits the following  presentation:
	$$H = \langle \sigma_1, \sigma_2 \mid \sigma_1\sigma_2\sigma_1 = \sigma_2\sigma_1\sigma_2,\; \sigma_1^3 = a \sigma_1^2 + b\sigma_1+ c,\;\sigma_2^3 = a \sigma_2^2 + b\sigma_2+ c \rangle,$$ with suitable $a,b,c\in R$.
	For the sake of brevity, we set $\sigma_1' = \sigma_1^{-1}$ and $\sigma_2' = \sigma_2^{-1}$.
	Let $z:= (s_1s_2)^3\in Z(W)$ and let $w: =s_1s_2s_1 = s_2s_1s_2 $. We have $z= (s_1s_2)^3 = (s_1s_2s_1)(s_2s_1s_2) = w^2$. We also denote by $W'$ the parabolic subgroup of $W$ generated by $s_1$ and by $H'$ the subalgebra of $H$ generated by $\sigma_1$.

	As we have already mentioned, the action of $W$ on the cosets of $W'$
	defines a graph with vertex set $\{W' x \mid x \in W\}$
	and edges $W' x \stackrel{u}{\longrightarrow} W'xu$, labelled by $u\in\{s_1, s_2\}$.
	The following diagram shows this action graph, except for a few edges:
	
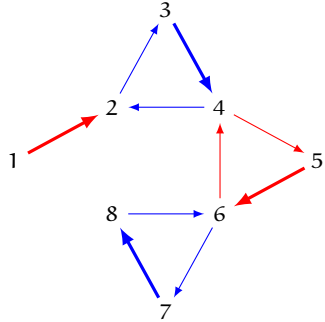
\begin{center}
		\begin{tikzpicture}
			\def \n {4}
			\foreach \s in {1,3,...,8}
			\node (\s) at ({180 - (180/\n * (\s-1))}:2) {$_{\s}$};
			\foreach \s in {2,4,...,8}
			\node (\s) at ({135 - (180/\n * (\s-2))}:1) {$_{\s}$};
			
			\foreach \s in {1,5} {
				\pgfmathtruncatemacro{\t}{1+\s}
				\pgfmathtruncatemacro{\a}{2+\s}
				\pgfmathtruncatemacro{\b}{3+\s}
				\draw[r,very thick] (\s) -- (\t);
				\draw[b] (\t) -- (\a);
				\draw[b] (\b) -- (\t);
				\draw[b,very thick] (\a) -- (\b);
			}
			\draw[r] (4) -- (5);
			\draw[r] (6) -- (4);
		\end{tikzpicture}
	\end{center}
	
	Here, instead of labeling the edges with $u$, we use the colours red and blue. More precisely, blue edges belong to generator $s_1$ and  red edges belong to generator $s_2$. Fat edges indicate a spanning tree of the coset graph.
	The vertices are labeled $_1, _2, \dots, _8$, corresponding to
	coset representatives $x_1=1,\; x_2=s_2,\; x_3=s_2s_1s_2,\; x_4=s_2s_1s_2s_1,\;x_5=zx_1,\; x_6=zx_2,\; x_7=zx_3,\; x_8=zx_4$. We make this choice of representatives as follows:
	We choose $x_1 = w^0$, $x_3 = w^1$, $x_5 = w^2$ and $x_7 = w^3$
	as anchor coset representatives  and we pick representatives for the
	remaining cosets along the spanning tree: $x_{2\ell} = x_{2\ell-1}u, \; \ell=1,2,3,4$ with generator
	$u = s_1, s_2$ as in the edge connecting coset $_{2\ell-1}$ to coset $_{2\ell}$ in the
	coset graph.
	
	By sending $s_1$ and $s_2$ to $\sigma_1$ and $\sigma_2$, respectively we can obtain from the elements $x_i$, $z$ and $w$ corresponding elements inside the Hecke algebra, which we also denote by $x_i$, $z$ and $w$.  The elements $z$ and $w$ are also central elements in $H$ (see \cite[Theorem 12.8]{bes} and  \cite[Theorem 2.2.4]{BMR}).
	
	As we have explained earlier, in order to prove that the set $\{x_i,\;, i = 1, \dots, 8\}$ is a basis of $H$ as $H'$-module  we construct the following coset table.
	
	\[
	\begin{array}{|l||c|c|c|c|}
		\hline
		x_i,\;i=1,\dots,8 & x_i.\sigma_1 & x_i.\sigma_2 & x_i.\sigma_1' & x_i.\sigma_2' \\ \hline \hline
		x_1 = z^0  & \sigma_1 \cdot x_1 & x_2 & \sigma_1' \cdot x_1 & \\
		x_2 = z^0 \sigma_2 & \sigma_1' \cdot x_3 & & x:\sigma_1' & x_1 \\
		x_3 = z^0 \sigma_2\sigma_1\sigma_2 & x_4 &  \sigma_1 \cdot x_3 &\sigma_1 \cdot x_2 & \sigma_1' \cdot x_3 \\
		x_4 = z^0 \sigma_2\sigma_1\sigma_2\sigma_1 & x_4:\sigma_1 & \sigma_1' \cdot x_5 & x_3 & x_4:\sigma_2' \\
		\hline
		x_5 = z^1 & \sigma_1 \cdot x_5 & x_6 & \sigma_1' \cdot x_5& \sigma_1 \cdot x_4 \\
		x_6 = z^1 \sigma_2 & \sigma_1' \cdot x_7 & x_6:\sigma_2 & x_6:\sigma_1' & x_5\\
		x_7 = z^1 \sigma_2\sigma_1\sigma_2  & x_8 & \sigma_1 \cdot x_7 & \sigma_1 \cdot x_6 & \sigma_1' \cdot x_7 \\
		x_8 = z^1\sigma_2\sigma_1\sigma_2\sigma_1 & x_8:\sigma_1 & & x_7 & \\ \hline
	\end{array}
	\]
	In this table, some notation is used as short-hand for more complex expressions. More precisely:
	\begin{itemize}
		\item $x_i : u$ indicates that $x_i.u$ can be computed from other entries in
		the table by using the relation $u = a + bu' + cu'u'$. This relation is obtained from the positive Hecke relation \eqref{ph} if we multiply both sides with $u'u'$. We see how we use such a relation in the example of  $x_4.\sigma_1$. We have:
		$x_4.\sigma_1=ax_4+bx_4.\sigma_1'+cx_4.\sigma_1'\sigma_1'$. We now notice that $x_4.\sigma_1'=x_3$. Therefore, $x_4.\sigma_1=ax_4+bx_3+cx_3.\sigma_1'$. Since $x_3.\sigma_1'=\sigma_1.x_2$ we have  that $x_4.\sigma_1=ax_4+bx_3+c\sigma_1.x_2$.
		\item Similarly, $x : u'$ indicates that $x.u'$ can be computed by
		using the inverse Hecke relation \eqref{invhecke} $u' = c^{-1}u^2 - ac^{-1}u - bc^{-1}$.
	\end{itemize}
	
	It remains now to compute the entries $x_1.\sigma_2'$, $x_2.\sigma_2$, $x_8.\sigma_2$ and
	$x_8.\sigma_2'$. We compute these  four excluded cases as follows, using braid relations and existing entries from the coset table.
	$$\begin{array}{lcl}
		x_2.\sigma_2 &=& x_2.(\sigma_2\sigma_1\sigma_2)\sigma_2'\sigma_1'=x_2.(\sigma_1\sigma_2\sigma_1)\sigma_2'\sigma_1' = x_4.\sigma_2'\sigma_1' \smallbreak\smallbreak\\
		x_8.\sigma_2' &=& x_8.(\sigma_2'\sigma_1'\sigma_2')\sigma_2\sigma_1=x_8.(\sigma_1'\sigma_2'\sigma_1')\sigma_2\sigma_1
		= x_6.\sigma_2\sigma_1\smallbreak\smallbreak\\
		x_1.\sigma_2' &=& x_1 : \sigma_2'\smallbreak\smallbreak \\
		x_8.\sigma_2 &=& x_8 : \sigma_2
	\end{array}$$
	
	The completed coset table and the fact that $H'$ is a free $R$-module  with basis $\{1,\;\sigma_1,\,\sigma_1^2\}$ proves the following:
	\begin{Proposition}
		With the above notation
		\begin{enumerate}[leftmargin=*]
			\item[(i)] $H$ is a free $H'$-module with basis $\{x_i \;:\; i = 1,\dots, 8\}$.
			\item[(ii)] $H$ is a free $R$-module with basis $\{b_j\;:\; j = 1,\dots, 24\} = \{x_i,\; \sigma_1 x_i,\; \sigma_1^2 x_i \;:\; i = 1,\dots, 8\}$.
		\end{enumerate}
	\end{Proposition}
	We will now express every element of the Hecke algebra as $R$-linear combination of the elements of $\{b_j\;:\; j = 1,\dots, 24\}$, using the following representation  for $H$ as a free module over the subalgebra~$H'$, which comes from the completed coset table:
	\begin{align*}
		\rho(\sigma_1)&:=\small{\left[
			\begin{array}{cccc|cccc}
				\sigma_1 & . & . & . & . & . & . & . \\
				. & . & \sigma_1' & . & . & . & . & . \\
				. & . & . & 1 & . & . & . & . \\
				. & c\sigma_1 & b & a & . & . & . & . \\ \hline
				. & . & . & . & \sigma_1 & . & . & . \\
				. & . & . & . & . & . & \sigma_1' & . \\
				. & . & . & . & . & . & . & 1 \\
				. & . & . & . & . & c\sigma_1 & b & a \\
			\end{array}
			\right]}
		\\
		\rho(\sigma_2)&:= \small{\left[
			\begin{array}{cccc|cccc}
				. & 1 & . & . & . & . & . & . \\
				. & . & -bc^{-1} & . & -ac^{-1} \sigma_1'' & -bc^{-2}\sigma_1' & -a c^{-2} \sigma_1'' & c^{-2} \sigma_1'' \\
				. & . & \sigma_1 & . & . & . & . & . \\
				. & . & . & . & \sigma_1' & . & . & . \\ \hline
				. & . & . & . & . & 1 & . & . \\
				. & . & . & c\sigma_1 & b & a & . & . \\
				. & . & . & . & . & . & \sigma & . \\
				c^3\sigma_1^2 & bc^2\sigma_1^2 & bc^2+b^2c\sigma_1 & bc\sigma_1^2 & -a^2c+b^2\sigma_1 & ac & -a^2 + a\sigma_1 & a \\
			\end{array}
			\right]}
	\end{align*}
	where $\sigma_1'' = \sigma_1^{-2}$. We recall that $\sigma' = \sigma^{-1} =  -bc^{-1} - ac^{-1}\sigma - c^{-1}\sigma^2$ and, hence, $\sigma'' =  -bc^{-1}\sigma^{-1} - ac^{-1} - c^{-1}\sigma_1= -bc^{-1} (-bc^{-1} - ac^{-1}\sigma - c^{-1}\sigma^2) - ac' - c'\sigma_1$. Therefore, the entries of the matrices $\rho(\sigma_1)$ and $\rho(\sigma_2)$ involve only positive powers of $\sigma_1$.

	In the following example, one can see how we can use $\rho$ in order to express every element of the Hecke algebra as an $R$-linear combination of elements in the basis $\{b_j\;:\; j = 1,\dots, 24\} = \{x_i,\; \sigma_1 x_i,\; \sigma_1^2 x_i \;:\; i = 1,\dots, 8\}$.
	$$\begin{array}{lcl}
		\sigma_1^2 \sigma_2^2
		&=& (1_{H'}, 0,0,0,0,0,0,0)\cdot\rho(\sigma_1)\cdot \rho(\sigma_1)\cdot\rho(\sigma_2) \cdot\rho(\sigma_2) \smallbreak\smallbreak\smallbreak\\
		&=& (
		0,\, 0,\, - bc^{-1} \sigma_1^2,\, 0,\,  -ac^{-1},\, - bc^{-2} \sigma_1,\, - ac^{-2},\, c^{-2})\smallbreak \smallbreak\smallbreak\\
		&=&
		0 x_1 + 0 x_2 - bc^{-1} \sigma_1^2 x_3 +  0 x_4  -ac^{-1} x_5 - bc^{-2} \sigma_1 x_6 - ac^{-2} x_7 + c^{-2} x_8\smallbreak\smallbreak\smallbreak\\
		&=&
		- bc^{-1} b_9   -ac^{-1} b_{13} - bc^{-2} b_{17} - ac^{-2} b_{19} + c^{-2} b_{22}.
	\end{array}$$
	\qed
\end{ex}

\section{Center}
\label{sec:center}

Let $W$ be a complex reflection group and let $H$ be the associated Hecke algebra, defined over the Laurent polynomial ring $R=\mathbb{Z}[u_{s,j},\;u_{s,j}^{-1}]$. We recall that  $F$ denotes  the splitting  field $\mathbb{C}(v_{s,j})$, with the indeterminates $v_{s,j}$ as defined in \eqref{split}. We also recall that $FH$ denotes the algebra $H\otimes_{R}F$.

As we have mentioned in section \ref{table}, we have the validity of the BMR freeness conjecture \ref{BMR free}. As a result, we can fix
a basis $\mathcal{B}=\{b_w : w \in W\}$ of $H$, indexed by the elements of $W$.
We now assume also the validity of the BMM symmetrising trace conjecture \ref{BMM sym}, meaning that $H$ admits a unique symmetrising trace $\tau$. We denote by
$\mathcal{B^{\vee}}=\{b_w^{\vee} : w \in W\}$ the dual basis of $\mathcal{B}$ with respect to $\tau$, uniquely determined by the condition $\tau(b_{w_1}b_{w_2}^{\vee})=\delta_{{w_1},{w_2}}$ for all $w_1,w_2\in W$.

We denote now by $\Cl(W)$ the set of conjugacy classes of elements of $W$.  For each class $C \in Cl(W)$, we choose a representative  $w_C \in C$. The square matrix $\bigl(\chi(b_w)\bigr)_{\chi, C}$
of character values is invertible as it specializes to the character
table of $W$.  Hence, for each $w \in W$, there are uniquely determined
coefficients $f_{w,C}\in F$ by the condition
\begin{align*}
	\chi(b_w) = \sum_{C \in \Cl(W)} f_{w, C}\, \chi(b_{w_C})
\end{align*}
for all $\chi \in \Irr(FH)$. These coefficients depend on the choice of the elements $w_C$ and of the basis $\mathcal{B}$.

\begin{rem}\label{cc} Let $W$ be a real reflection group. Then, we choose $\mathcal{B}$ to be the standard basis $\{T_w:w\in W\}$ and $w_C$  an element of minimal length in $C$. We know
	\cite[\S 8.2.2 and \S 8.2.3]{GP} that the coefficients $f_{w, C}$ are  independent of the actual choice
	of the elements $w_C$ and they belong to $R$. In this case, $f_{w, C}$ are known as \emph{class polynomials}.
\end{rem}
The following theorem has been shown by Geck and Rouquier in the real case (see \cite[\S 5.1]{GeckRouquier97} or \cite[Theorem 8.2.3 and Corollary 8.2.4]{GP}). However, the proof in general for every complex reflection group is slightly different from the original, as the prior existence
of class polynomials in the complex case cannot be assumed.

\begin{Theorem}\label{T1} Let $W$ be a complex reflection group. The elements
	\[
	y_C = \sum_{w \in W} f_{w, C}\, b_w^{\vee}, \quad C \in \Cl(W),
	\]
	form a basis of the center $Z(FH)$.
\end{Theorem}

\begin{proof}
	For each class $C\in\Cl(W)$, we define the function $f_C \colon H \to F,\, b_w \mapsto f_{w, C}$. We first prove that $f_C$ is a trace function.
	For each $b_w\in \mathcal{B}$ and for each $\chi \in \Irr(FH)$ we have:
	\begin{align*}
		\chi(b_w)
		= \sum_{C\in\Cl(W) } f_{w, C}\, \chi(b_{w_C})
		= \sum_{C\in\Cl(W)} f_C(b_w)\, \chi(b_{w_C})
		= \sum_{C\in\Cl(W)} \chi(b_{w_C})\, f_C(b_w)\text.
	\end{align*}
	Therefore, \begin{equation} \label{eq1}
		\chi = \sum_{C\in\Cl(W)} \chi(b_{w_C}) f_C,\; \text{ for all } \chi \in \Irr(FH).
	\end{equation}
	The matrix $(\chi(b_{w_C}))_{\chi,C}$ specializes to the character table of $W$, which
	has a nonzero determinant and, hence, it is invertible. We denote its inverse by
	$(i(C, \chi))_{\chi,C}$. Therefore, it follows from (\ref{eq1}) that  \begin{align*}
		f_C = \sum_{\chi \in \Irr(FH)} i(C, \chi)\, \chi,
	\end{align*}
	which proves that $f_C$ is a trace function.
	
	We now prove that the set $\{f_C\}_C$ is a basis of the space of trace functions on $FH$. Since the algebra $FH$ is split semisimple, we know \cite[Exercise 7.4(b)]{GP} that the set $\Irr(FH)$ is
	a basis for the space of trace functions on $FH$. Therefore, from Equation (\ref{eq1}) we conclude that $\{f_C\}_C$ is a spanning set of this space. We now prove that the set $\{f_C\}_C$ is a linearly independent set of trace functions. For this purpose, it suffices to prove that $f_C(b_{w_{C'}}) = \delta_{C,C'}$. By definition, $f_C(b_{w_{C'}})=f_{w,C'}$ and $f_{w,C'}$ is the coefficient of the column of the character value $\chi(b_{w_{C'}})$,
	when the column $\chi(b_w)$ is expressed as a linear combination of the character values $\chi(b_{w_C})$. We have
	\begin{align*}
		\chi(b_{w_C}) = 1\cdot\chi(b_{w_C}) + 0,
	\end{align*}
	therefore $f_{w,C'}=0$, for all $C'\not=C$ and $f_{w,C} = 1$. We conclude that the set $\{f_C\}_C$ is a basis of the space of trace functions on $FH$.
	
	We now prove that the set $\{y_C\}_C$ is a basis of the center $Z(FH)$. Since we assume that the algebra $FH$ admits a symmetrising trace, we can apply \cite[Lemma 7.1.7]{GP}, which states that the set $\{f_C^*\}_C$ is a basis of $Z(FH)$,  where $f_C^*$ denotes the dual of $f_C \in \Hom_{F}(FH, F)$. We have:
	\begin{align*}
		f_C^* = \sum_w f_C(b_w) b_w^{\vee} = \sum_w f_{w,C} b_w^{\vee} = y_C
	\end{align*}
	Therefore,  the set $\{y_C\}_C$ is a basis of the center $Z(FH)$.
\end{proof}

We now prove a \emph{dual} version of Theorem \ref{T1}.  For each class $C \in Cl(W)$, we choose again a representative  $w_C \in C$ and we define
coefficients $g_{w,C} \in F$ by the condition
\begin{align*}
	\chi(b_w^{\vee}) = \sum_{C \in \Cl(W)} g_{w, C}\, \chi(b_{w_C}^{\vee}),\;\text{ for all } \chi \in \Irr(FH).
\end{align*}

\begin{Theorem}\label{T2}
	The elements
	\[
	z_C = \sum_{w \in W} g_{w, C}\, b_w, \quad C \in \Cl(W),
	\]
	form a basis of the center $Z(FH)$.
\end{Theorem}

\begin{proof}
	For each class $C\in\Cl(W)$, we define the function $g_C \colon H \to F,\, b_w^{\vee} \mapsto g_{w, C}$. As in the proof of Theorem \ref{T1},
	we prove that $\{g_C\}_{C}$ is a basis of the space of trace functions on $FH$.
	For each $b_w\in \mathcal{B}$ and for each $\chi \in \Irr(FH)$ we have:
	\begin{align*}
		\chi(b_w^{\vee})
		= \sum_{C \in \Cl(W)} g_{w, C}\, \chi(b_{w_C}^{\vee})
		= \sum_C g_C(b_w^{\vee})\, \chi(b_{w_C}^{\vee})
		= \sum_C \chi(b_{w_C}^{\vee})\, g_C(b_w^{\vee})
	\end{align*}
	Therefore, \begin{equation} \label{eq2}
		\chi = \sum_{C\in\Cl(W)}\chi(b_{w_C}^{\vee}) g_C,\; \text{ for all } \chi \in \Irr(FH).
	\end{equation}
	The matrix $(\chi(b_{w_C}^{\vee}))_{\chi,C}$ is invertible, since
	it specializes to the character table of $W$, which
	has a nonzero determinant. We denote its inverse by
	$(j(C, \chi))_{\chi,C}$ and Equation (\ref{eq2}) becomes:
	\begin{align*}
		g_C = \sum_{\chi\in \Irr(FH)} j(C, \chi)\, \chi
	\end{align*}
	and, hence, $g_C$ is a trace function.
	Using the same arguments again as in the  proof of Theorem \ref{T1}, we have $g_C(b_{C'}^{\vee}) = \delta_{C,C'}$ and, hence, together with Equation (\ref{eq2}) we conclude that  the set $\{g_C\}_C$ is a basis of the space of trace functions on $FH$.
	
	We now prove that the set $\{z_C\}_C$ is a basis of the center $Z(FH)$. Since we assume that the algebra $FH$ admits a symmetrising trace, we can apply \cite[Lemma 7.1.7]{GP}, which states that the set $\{g_C^*\}_C$ is a basis of $Z(FH)$,  where $g_C^*$ denotes the dual of $g_C \in \Hom_{F}(FH, F)$. We have:
	\begin{align*}
		g_C^* = \sum_w g_C(b_w^{\vee}) (b_w^{\vee})^{\vee} = \sum_w g_{w,C} b_w = z_C
	\end{align*}
	
	Therefore,  the set $\{z_C\}_C$ is a basis of the center $Z(FH)$.
\end{proof}

Note that the elements $z_C$ depend on the choice of the basis $\mathcal{B}$ and of the class representatives $w_C$,
$C \in \Cl(W)$.

We focus now on the exceptional groups we have described in section \ref{table}, namely the groups  $G_n$, where $n\in\{4,\dots,8,12,13,22\}$.
In section \ref{table} we have made a particular choice of a basis $\mathcal{B}$ and we managed with the help of the coset table to express each element of the Hecke algebra as linear combination of elements in $\mathcal{B}$. In particular, we can express any product $bb'$, with $b,b'\in \mathcal{B}$ as a linear combination of the elements in  $\mathcal{B}$. We use this linear combination in order to give another proof the BMM symmetrising trace conjecture \ref{BMM sym} (this conjecture is known to hold for these complex reflection groups, as we have mentioned in section \ref{Hecke}).

We define a linear map $\tau\colon H  \to R$ by setting $\tau(\sum_{b\in \mathcal{B}} \alpha_b b) = \alpha_1$. We can now calculate the Gram matrix $A=(\tau(bb'))_{b,b'\in \mathcal{B}}$ and prove the following:

\begin{Proposition}\label{tracef} Let $W$ be one of the groups $G_n$,  where $n \in \{4,\dots,8,12,13,22\}$. With the above choice of $\mathcal{B}$, we have:
	\begin{enumerate}[leftmargin=*]
		\item[(i)] 	The matrix $A$ is symmetric and its determinant is a unit in $R$.
		\item[(ii)] Condition \eqref{extra2} is satisfied.
	\end{enumerate}
\end{Proposition}
We note here that in the project's webpage \cite{pr} we give explicitly for each group the determinant of the Gram matrix $A$.

Our goal is to find the coefficients $g_{w,C}$. By their definition, we need to find the dual basis $\mathcal{B}^{\vee}$. More precisely, we need to express each $b_i^{\vee}\in \mathcal{B}^{\vee}$ as $R$-linear combination in elements of $\mathcal{B}$. Let
$b_i^{\vee}=a^i_1b_1+a^i_2b_2+\dots+a^i_{|W|}b_{|W|}$ be this linear combination. Then, we have the following:
$$\begin{pmatrix}
	\tau(b_i^{\vee}b_1)\smallbreak\smallbreak\smallbreak\\
	\tau(b_i^{\vee}b_2)\smallbreak\smallbreak\\
	\vdots\smallbreak\smallbreak\\
	\tau(b_i^{\vee}b_{|W|})\\
\end{pmatrix}
=A\begin{pmatrix}
	a_1^i\smallbreak\smallbreak\smallbreak\\
	a_2^i\smallbreak\smallbreak\\
	\vdots\smallbreak\smallbreak\\
	a_{|W|}^i\\
\end{pmatrix}
$$
By Proposition \ref{tracef}(i) $A$ is invertible in $R$. Moreover, we know that  the dual basis is determined by the condition  $\tau(b_k^{\vee}b_j) = \delta_{kj}$. Therefore:
$$\begin{pmatrix}
	a_1^i\smallbreak\smallbreak\smallbreak\\
	a_2^i\smallbreak\\
	\vdots\smallbreak\\
	a_{|W|}^i\\
\end{pmatrix}=A^{-1}\;\;
\begin{blockarray}{(>{\:}c<{\:})l}
	0&\\
	\vdots&\\
	1&\,\,i-\text{place}\\
	\vdots&\\
	0&\\
\end{blockarray} = i \text{-th column of } A^{-1}$$
Having now  expressed each $b_i^{\vee}\in \mathcal{B}^{\vee}$ as $R$-linear combination in elements of $\mathcal{B}$, it remains to make a choice of class representatives $w_C$, which will allow us to calculate the coefficients $g_{w,C}$. We give for each group this particular choice of representatives in the webpage of this project \cite{pr}. The following result is the main theorem of this paper.

\begin{Theorem}Let $W$ be the exceptional group $G_n$, where $n\in\{4,\dots, 8, 12, 13, 22\}$. There exists a choice of a basis $\{b_w : w \in W\}$ of the Hecke algebra
	$H$ of $W$ and a choice of conjugacy class representatives $\{w_C : C \in \Cl(W)\}$, such that the  coefficients
	$g_{w, C}$ belong to $R$ and, hence, the set $\{z_C : C \in \Cl(W)\}$ is a basis of $Z(H)$.
\end{Theorem}

For a better understanding of the above theorem, we revisit the example of $G_4$.
\begin{ex}
	Let $W$ be the complex reflection group $G_4$. In the example \ref{exa} we saw that the associated Hecke algebra $H$ admits the basis $$\{b_{3k+m}\;:\; k = 0,\dots, 7,\,m=1,2,3\},$$ 
	defined as $b_{3k+m}=\sigma_1^{m-1}x_{k+1}$, where $x_1=1,\; x_2=\sigma_2,\; x_3=\sigma_2\sigma_1\sigma_2,\; x_4=\sigma_2\sigma_1\sigma_2\sigma_1$, $x_5=zx_1,\; x_6=zx_2,\; x_7=zx_3,\; x_8=zx_4$. We now make the following choice for the class representatives $w_C$:
	$$w_{C_1}=b_1,\; w_{C_2}=b_{10},\;w_{C_3}=b_{13},\; w_{C_4}=b_{15},\;w_{C_5}=b_{22},\;w_{C_6}=b_{23},\; w_{C_7}=b_{24}.$$
	The evaluation of the irreducible characters of $H$ on the dual basis $\{b_i^{\vee}\}$
	yields the  elements $g_{i,C}$ and hence an explicit
	basis of $Z(FH)$:
	\begin{align*}
		z_1 &= ac^3(b_{3} + b_{5}) + (abc^2+c^3)(b_{6} + b_{7}) + (bc^2+ab^2c-a^2c^2)b_{8}
		+ ac^2(b_{9} + b_{11}) \\& + (abc+c^2)b_{12} + (2bc+ab^2-a^2c)b_{14}
		+ c(b_{18} + b_{19}) + b\,b_{20} + b_{24} \\
		z_2 &= c^3b_{3} + bc^2(b_{6} + b_{7}) + b^2c\,b_{8} + bc\,b_{12} + (-ac+b^2)b_{14} + c\,b_{17} - a\,b_{20} \\& + b_{21} + b_{23} \\
		z_3 &= c(b_{14} + b_{16}) - a\,b_{19} + b_{20} + b_{22} \\
		z_4 &= c^2 b_{5} - ac\,b_{8} + c(b_{9} + b_{11}) -a\,b_{14} + b_{15} \\
		z_5 &= b_{13} \\
		z_6 &= c(b_{2} + b_{4}) -a\,b_{7} + b_{8} + b_{10} \\
		z_7 &= b_{1} \\
	\end{align*}
	As we can see, the elements $g_{i,C}$ are actually elements in $R$ and, hence, the set $\{z_1,\dots, z_7\}$ is indeed a basis of $Z(H)$.
	\qed
\end{ex}

Based on these examples and the fact that complex reflection groups generalize the properties of real reflection groups, we believe that  one can find in this way a basis  $\{z_C : C \in \Cl(W)\}$ of $Z(H)$ for each complex reflection group $W$. Therefore, we state the following conjecture:

\begin{conj}
	Let $W$ be a complex reflection group. There exists a choice of a basis $\{b_w : w \in W\}$ of the Hecke algebra
	$H$ of $W$, and a choice of conjugacy class representatives $\{w_C : C \in \Cl(W)\}$ such that the construction of Theorem \ref{T2} yields polynomial coefficients
	$g_{w, C}\in R$, and hence a basis $\{z_C : C \in \Cl(W)\}$ of $Z(H)$.
\end{conj}
Note added in proof: We mention here that, after this work was completed, Hu and Shi \cite{hs} proved the aforementioned conjecture for the groups of type $G(d, 1, n)$. 

%\bibliographystyle{amsplain}
%\bibliography{paper}

\begin{thebibliography}{10}
	
	\bibitem{pr}
	\emph{The project's webpage}, \url{https://www.eirinichavli.com/center.html}.
	
	\bibitem{bes}
D.~Bessis, \emph{Finite complex reflection arrangements are $K(\pi, 1)$}, Ann.~Math.~{\bf 181}(3) (2015), 809--904.

	\bibitem{BCC} C.~Boura, E.~Chavli, M.~Chlouveraki,  \emph{The BMM symmetrising trace conjecture for the exceptional 2-reflection groups of rank 2}, J.~Algebra {\bf 558} (2020), 176--198.

	\bibitem{BCCK}
C.~Boura, E.~Chavli, M.~Chlouveraki, K.~Karvounis,  \emph{The BMM symmetrising trace conjecture for groups $G_4$, $G_5$, $G_6$, $G_7$, $G_8$}, J.~ Symbolic Comput.~ {\bf 96} (2020), 62--84.
	
	\bibitem{Bou05}
	N.~Bourbaki, Lie groups and Lie algebras, chap. 4--6, Elements of mathematics, English translation of ``Groupes et alg\`ebres de Lie'', Springer, 2005.

	
	\bibitem{BMM}
 M.~Brou{\'e}, G.~Malle, J.~Michel, \emph{Towards Spetses I}, Trans.~Groups {\bf 4} (1999), 157--218.
	
	\bibitem{BMR}
 M.~Brou{\'e}, G.~Malle, R.~Rouquier, \emph{Complex 
	reflection groups, braid groups, Hecke algebras}, J.~reine angew.
Math. {\bf 500} (1998), 127--190.


	
	\bibitem{CC}
	 E.~Chavli, M.~Chlouveraki, \emph{The freeness and trace conjectures
		for parabolic {H}ecke subalgebras}, J.~Pure and App.~ Alg. {\bf 226}(10) (2022): 107050.
	
	\bibitem{fr}
	A.~Francis, \emph{Centralizers in the {H}ecke algebras of complex
		reflection groups}, arXiv preprint: 0707.2822 (2007).
	
\bibitem{chevie}
		M.~Geck, G.~Hi{\ss}, F.~L{\"u}beck, G.~Malle, G.~Pfeiffer, \emph{CHEVIE, A system for computing and processing generic
			character tables},
		Appl. Algebra Engrg. Comm. Comput. {\bf 7} (1996), 175--210.
	
	
	\bibitem{GP}
	M.~Geck, G.~Pfeiffer, \emph{Characters of finite {C}oxeter groups
		and {I}wahori-{H}ecke algebras}, London Math.~ Soc.~ Monographs, New Series 21, Oxford University Press, New York, 2000.
	
	\bibitem{GeckRouquier97}
	M.~Geck, R.~Rouquier, \emph{Centers and simple modules for
		{I}wahori-{H}ecke algebras},  In:
	M. Cabanes (Eds.), Finite reductive groups ({L}uminy, 1994), 251--272,
	Progr.~
	Math., {\bf 141}, Birkh\"{a}user Boston, Boston, 1997.
	
	\bibitem{hs}J.~Hu, L.~Shi, \emph{Proof of the Center Conjectures for the cyclotomic Hecke and KLR algebras of type $A$},  arXiv preprint: 2211.07069 (2022).
	
	\bibitem{malle1}
	 G.~Malle, \emph{On the rationality and fake degrees of characters of cyclotomic algebras}, J.~Math.~Sci.~Univ.
	Tokyo {\bf 6 }(1999), 647--677.
	
	\bibitem{MalleMichel}
	G.~Malle, J.~Michel, \emph{Constructing representations of Hecke algebras for complex
		reflection groups}, LMS J. Comput. Math. {\bf 13} (2010), 426--450.
	
	\bibitem{chevie-jm}
J.~Michel, \emph{The development version of the CHEVIE package of GAP3}, J.~Algebra {\bf 435} (2015), 308--336.
	
	\bibitem{ShTo}
G.~C.~Shephard, J.~A.~Todd, \emph{Finite unitary reflection
	groups}, Canad.~J.~Math. {\bf 6} (1954), 274--304.
\end{thebibliography}

\end{document}